\documentclass[11pt]{amsart}
\usepackage[utf8]{inputenc}

\usepackage[abbrev,lite,alphabetic]{amsrefs}
\usepackage[colorlinks,plainpages,urlcolor=blue,linkcolor=reference,citecolor=citation,colorlinks = true, hyperfootnotes=false]{hyperref}
\hypersetup{citecolor=gray, linkcolor = darkgray, urlcolor= gray}

\usepackage[T1]{fontenc}
\usepackage{etex}
\usepackage[shortlabels]{enumitem}
\usepackage{moreenum}
\usepackage{subfiles}
 \usepackage{lmodern}
\usepackage{wrapfig}
\usepackage{csquotes}
\usepackage{comment}

\usepackage{tikz-cd}
\usepackage{mathrsfs}
\usepackage{sseq}
\usepackage{mathdots}
\usepackage{thm-restate}
\usepackage{amsmath}
\usepackage{amsxtra}
\usepackage{amscd}
\usepackage{amsthm}
\usepackage{amsfonts}
\usepackage{amssymb}
\usepackage{mathtools}
\usepackage[scr=rsfs]{mathalfa}
\usepackage{eucal}
\usepackage{bbm}
\usepackage[all,cmtip]{xy}

\usepackage{graphicx}
\usepackage{tikz}
\usepackage{morefloats}
\usepackage{pdfpages}
\usetikzlibrary{matrix,arrows,decorations.pathmorphing,quotes,angles,decorations.markings,shapes.misc}
\RequirePackage{color}

\usepackage[colorinlistoftodos, textsize=tiny]{todonotes}
\usepackage[margin=1in,marginparwidth=0.8in, marginparsep=0.1in]{geometry}
\usepackage[framemethod=TikZ]{mdframed}

\usepackage{cleveref}

\definecolor{reference}{rgb}{0.20,0.36,0.74}
\definecolor{citation}{rgb}{0,.40,.80}

\usepackage{lipsum}
\usepackage{cleveref}
\crefname{section}{\S \!\!}{\S\S \!\!}
\crefname{equation}{}{}
\crefname{enumi}{}{}
\crefname{appendix}{\S \!\!}{\S\S \!\!}  

\setenumerate[1]{label={(\arabic*)}}

\allowdisplaybreaks

\theoremstyle{plain}
\newtheorem{theorem}{Theorem}[section]
\newtheorem{mainthm}{Theorem}

\newtheorem{mainconj}{Conjecture}

\newtheorem{proposition}[theorem]{Proposition}
\newtheorem{lemma}[theorem]{Lemma}
\newtheorem{corollary}[theorem]{Corollary}
\newtheorem*{corollary*}{Corollary}
\theoremstyle{definition}

\newtheorem{definition}[theorem]{Definition}

\newtheorem{notation}[theorem]{Notation}

\newtheorem{remark}[theorem]{Remark}
\newtheorem{example}[theorem]{Example}
\newtheorem*{example*}{Example}
\newtheorem*{remark*}{Remark}

\def\cA{\mathcal A}\def\cC{\mathcal C}\def\cD{\mathcal D}
\def\cF{\mathcal F}
\def\cI{\mathcal I}\def\cL{\mathcal L}
\def\cO{\mathcal O}
\def\cQ{\mathcal Q}\def\cS{\mathcal S}\def\cT{\mathcal T}

\def\cY{\mathcal Y}

\newcommand{\fL}{\mathfrak{L}}

\makeatletter
\providecommand{\leftsquigarrow}{%
  \mathrel{\mathpalette\reflect@squig\relax}%
}
\newcommand{\reflect@squig}[2]{%
  \reflectbox{$\m@th#1\rightsquigarrow$}%
}
\makeatother

\newcommand{\Fun}{{\sf Fun}}

\newcommand{\Spec}{{\sf Spec}}

\newcommand{\St}{{\sf St}}
\newcommand{\Mod}{{\sf Mod}}

\newcommand{\Perf}{{\sf Perf}}

\renewcommand{\lim}{{\sf lim}}

\newcommand{\id}{{\sf id}}

\newcommand{\Maps}{{\sf Maps}}

\newcommand{\CC}{\mathbb{C}}

\newcommand{\EE}{\mathbb{E}}

\newcommand{\LL}{\mathbb{L}}

\newcommand{\ZZ}{\mathbb{Z}}

\newcommand{\gr}{{\sf gr}}

\newcommand{\Aut}{\mathsf{aut}}
\newcommand{\End}{\mathsf{end}}

\newcommand{\Coh}{\mathsf{Coh}}

\newcommand{\QCoh}{\mathsf{QCoh}}

\newcommand{\IndCoh}{\mathsf{IndCoh}}

\newcommand{\Sph}{\mathsf{Sph}}

\newcommand{\Perv}{\mathsf{Perv}}
\newcommand{\PS}{2\mathsf{Perv}}
\newcommand{\Sh}{\mathsf{Sh}}
\newcommand{\Sing}{\mathsf{Sing}}

\newcommand{\Cone}{\mathsf{Cone}}
\newcommand{\oLL}{\overline{\mathbb{L}}}

\newcommand{\ul}{\underline}

\newcommand{\Loc}{\mathsf{Loc}}

\newcommand{\Ind}{{\sf Ind}}

\newcommand{\Tr}{{\sf Tr}}
\newcommand{\spectral}{\cL_B(\CC/\CC^\times)}
\newcommand{\Exit}{{\sf Exit}}
\newcommand{\Enter}{{\sf Enter}}
\newcommand{\Vect}{{\sf Vect}}
\newcommand{\univ}{L_{{\sf univ}}}

\newcommand{\kk}{\mathbb{C}}

\title{Betti Tate's thesis and the trace of perverse schobers}
\author{Benjamin Gammage and Justin Hilburn}

\begin{document}

\maketitle

\begin{abstract}
We propose a conjecture on the categorical trace of the 2-category of perverse schobers (expected to model the Fukaya-Fueter 2-category of a holomorphic symplectic space). By proving a Betti geometric version of Tate's thesis, and combining it with our previous 3d mirror symmetry equivalence and the Ben-Zvi--Nadler--Preygel result on spectral traces, we are able to establish our conjecture in the simplest interesting case. 
\end{abstract}

\section{Introduction}
Let $\cL_B(\CC/\CC^\times)$ be the Betti loop space $\Maps(S^1_B, \CC/\CC^\times)$ --- i.e., the moduli space of rank-1 Betti local systems on the punctured disk equipped with a flat section --- and let $\fL\CC=\CC((t))$ be the loop space of $\CC,$ which we understand as an (ind-infinite type) ind-scheme. Let $\cS$ be the stratification of $\fL\CC$ by order of zero or pole at $0\in \CC$ (where we understand the constant map to 0 as having a zero of infinite order at 0): in other words, $\cS$ is the stratification of the space $\CC((t))$ of Laurent polynomials by the valuation $\CC((t))\to \ZZ\cup \{\infty\}$ taking a Laurent polynomial to the degree of its lowest-order term. The following result is the main theorem of this paper.
\begin{mainthm}[``Betti Tate's thesis'']\label{mainthm:betti-tate}
There is an equivalence
\begin{equation}\label{eq:mainthm1}
    \Sh^!_\cS(\fL \CC) \simeq \IndCoh(\cL_B(\CC/\CC^\times))
\end{equation}
between the category of (co)sheaves on $\fL \CC$ with singular support in conormals to strata of $\cS$ and the category of ind-coherent sheaves on $\cY.$ This equivalence intertwines the structure of module categories for the monoidal category whose automorphic and spectral realizations are the left- and right-hand sides, respectively, of the ``Betti geometric class field theory'' equivalence
\begin{equation}\label{eq:betti-gcft}
    \Loc(\fL\CC^\times) \simeq \Loc(\ZZ\times \CC^\times) \simeq \IndCoh(B\CC^\times\times \CC^\times) \simeq \IndCoh(\cL_B(B\CC^\times)),
\end{equation}
where $\Loc$ denotes the category of (possibly infinite-dimensional) local systems, and the left- (resp. right-)hand side of \eqref{eq:betti-gcft} acts on the left- (resp. right-)hand side of \eqref{eq:mainthm1} by convolution (resp. pullback and tensor).
\end{mainthm}
A de Rham version of this theorem, in which the constructible-(co)sheaf category $\Sh_\cS^!$ is replaced by the category ${\sf D}^!$ of D-modules, and the Betti loop space $\cL_B(\CC/\CC^\times)$ is replaced by the de Rham loop space $\cL_{dR}(\CC/\CC^\times):=\Maps(\mathring{\mathbb{D}}_{dR}, \CC/\CC^\times)$ (with the categories in \Cref{eq:betti-gcft} similarly replaced by their de Rham analogues), was proved in \cite{HR21}. We refer to the introduction of that paper for more on the physical and mathematical background of this result, and to \cite{BZN-Betti} for more about the ``Betti'' approach to dualities in 3d and 4d topological field theories.

\subsection{2-categories and trace decategorification}
We would like to understand \Cref{mainthm:betti-tate} as ``reduction on a circle'' of the main theorem of \cite{GHMG}. Let $\LL\subset T^*\CC$ be the Lagrangian $\LL=\CC\cup T^*_0\CC$ given as the union of the zero-section and a cotangent fiber, and let $\oLL=\CC/\CC^\times \cup T^*_{0/\CC^\times}\CC/\CC^\times$ be its image in $T^*(\CC/\CC^\times).$
\begin{theorem}[\cite{GHMG}]\label{thm:3dmirror-basic}
There is an equivalence
\begin{equation}\label{eq:3dmirror-basic}
    2\Perv_\LL(\CC)\simeq 2\IndCoh_{\oLL}(\CC/\CC^\times)
\end{equation}
between the 2-category of perverse schobers on $\CC$ with singular support in $\LL$ and the 2-category of Ind-coherent sheaves of categories on $\CC/\CC^\times$ with singular support in $\oLL.$
\end{theorem}
We refer to \cite{GHMG} for more details about these 2-categories. For us, the important point will be that by definition, the right-hand 2-category has a definition as the 2-category $\Mod_\cA(\St)$ of module categories for the monoidal category
\begin{equation}\label{eq:basic-e1cat}
\cA:=\IndCoh\left((\CC/\CC^\times\sqcup 0/\CC^\times)\times_{\CC/\CC^\times}(\CC/\CC^\times\sqcup 0/\CC^\times)\right),
\end{equation}
where the monoidal structure is given by convolution.

We now recall the main result of \cite{BZNP}.
\begin{theorem}[\cite{BZNP}*{Theorem 1.2.12}]\label{thm:bside-traces}
Let $X\to Y$ be proper and surjective with $X,Y$ smooth, and let $\cA$ be the monoidal category $\IndCoh(X\times_Y X),$ with the monoidal structure given by convolution. Then there is an equivalence
\[
\Tr(\cA)\cong \IndCoh_{\Lambda_{X/Y}}(\cL_BY)
\]
of $S^1$-categories between the categorical trace of the monoical category $\cA,$ and the category of Ind-coherent sheaves on the Betti-loop space $\cL_BY$ with a certain singular support condition $\Lambda_{X/Y}.$
\end{theorem}
The singular support condition $\Lambda_{X/Y}$ of \Cref{thm:bside-traces} is defined as follows: Let $\Sing(X/Y)=T^{*-1}(X\times_Y X)$ be the degree-$(-1)$ part of the cotangent bundle of $X\times_Y X,$ and define $\Lambda_{X/Y}\subset T^{*-1}\cL_BY$ to be the push-pull of $\Sing(X/Y)$ along the correspondence
\[
\xymatrix{
X\times_Y X&\ar[l]
(X\times_YX)\times_{X\times X}X\simeq X\times_{Y\times X}X
\ar[r] & Y\times_{Y\times Y}Y\simeq \cL_BY.
}
\]
It follows that in the case where $X = \CC/\CC^\times\sqcup 0/\CC^\times \to Y = \CC/\CC^\times$ is the map used to define the monoidal category $\cA$ of \eqref{eq:basic-e1cat}, the singular support $\Lambda_{X/Y}$ is as large as possible, namely all of $T^{*-1}\cL_B(\CC/\CC^\times).$ Therefore, rephrasing statements about $\cA$ in the language of $2\IndCoh,$ we may relate the right-hand sides of the equivalences \eqref{eq:mainthm1} and \eqref{eq:3dmirror-basic}:
\begin{corollary}\label{cor:bside-trace}
The categorical trace $\Tr(2\IndCoh_{\oLL}(\CC/\CC^\times))$ of the 2-category of \eqref{eq:3dmirror-basic} is equivalent to the category $\IndCoh(\cL_B(\CC/\CC^\times))$ of Ind-coherent sheaves on the Betti loop space of $\CC/\CC^\times.$
\end{corollary}

\subsection{A-side traces}
\Cref{mainthm:betti-tate}, together with \Cref{cor:bside-trace}, is important to us because it gives evidence that there should be an ``A-model'' version of \Cref{thm:bside-traces}, which we interpret as a theorem about the trace of the 2-category of Ind-coherent sheaves on $Y$ with a given singular support condition. The philosophy of 3d mirror symmetry, as laid out in \cite{GHMG}, suggests that such a B-model (i.e., coherent) 2-category should be equivalent to an A-model (i.e., symplectic) 2-category associated to some dual space, as is the case in \Cref{thm:3dmirror-basic}. The 2-category of perverse schobers which appears in that theorem is so far very mysterious and so far is defined only in very special cases, but we would like to have an analogous theory for describing its categorical trace. The main conjecture of this paper is the following:

\begin{mainconj}\label{mainconj:perv-traces}
Let $X$ be an algebraic variety or stack, and let $\LL\subset T^*X$ be a conic Lagrangian. There is an equivalence of $S^1$-categories
\begin{equation}\label{eq:conjecture}
\Tr(2\Perv_{\LL}(X))\simeq \Sh_{\Lambda(\LL)}(\fL X)
\end{equation}
between the categorical trace of the 2-category of perverse schobers on $X$ with singular support in $\LL$ and the category of sheaves on the loop space $\fL X$ with singular support in a certain Lagrangian\footnote{See \Cref{subsec:remarks} for an approach to construction of the Lagrangian $\Lambda(\LL).$} $\Lambda(\LL)\subset T^*\fL X$ determined by $\LL.$
\end{mainconj}

Unfortunately, it is not currently possible to prove \Cref{mainconj:perv-traces}, and indeed it is barely even possible to {\em state} this conjecture, given that in almost no cases is the 2-category of perverse schobers even defined. Nevertheless, in the case considered in \cite{GHMG}, where $(X,\LL) = (\CC,\CC\cup T^*_0\CC),$ the 2-category $2\Perv_\LL(X)$ does admit a definition, namely as the 2-category $\Sph$ of {\em spherical adjunctions}, and in this case the combination of \Cref{mainthm:betti-tate}, \Cref{thm:3dmirror-basic}, and \Cref{cor:bside-trace} allows us to conclude that the conjecture is true in this case:

\begin{mainthm}\label{mainthm:example}
Let $2\Perv(\CC,0)$ be the 2-category of perverse schobers on the stratified space $\CC=\CC^\times\sqcup 0,$ which by definition is the 2-category $\Sph$ of spherical adjunctions. Then there is an equivalence
\[
\mathsf{Tr}(2\Perv(\CC,0)) \simeq \Sh^!_\cS(\fL \CC)
\]
between the categorical trace of $\PS(\CC,0)$ and the category of cosheaves on $\cL\CC$ defined above.
\end{mainthm}

\begin{remark}
One can also prove a version of \Cref{mainthm:example} describing the center, rather than trace, of the 2-category $2\Perv(\CC,0)$ (or rather its compact objects). On the B-side, this uses an analogue of \Cref{thm:bside-traces} identifying the center of the monoidal category $\Coh(X\times_YX)$ (with notation as in \Cref{thm:bside-traces}) with the category $\Coh_{prop/Y}(\cL_BY)$ of coherent sheaves on $\cL_BY$ whose pushforward to $Y$ is coherent. Under the equivalence \Cref{eq:mainthm1}, this finiteness condition corresponds to the requirement on $\Sh^!_\cS(\fL \CC)$ that objects should have finite-rank (co)stalks. 

Thus the analogue of \Cref{mainthm:example} for centers identifies the center of the category of compact objects in $2\Perv(\CC,0)$ with the subcategory of $\Sh^!_\cS(\fL \CC)$ of objects whose (co)stalks are all finite rank; moreover, this identification should be upgraded to an equivalence of $\EE_2$-categories, where the $\EE_2$-structure on the sheaf category comes from a pair-of-pants correspondence among loop spaces.
In slightly modified form, such an $\EE_2$-category (and an equivalence with the center of the B-side 2-category) was described by \cite{Ballin-Niu} in the setting of conformal field theory.
\end{remark}

\subsection{Further remarks}\label{subsec:remarks}
We conclude the introduction with a brief explication of, and context for, \Cref{mainthm:betti-tate} and \Cref{mainconj:perv-traces}.
The latter conjecture sits as part of a long tradition of relating Hochschild-type invariants of (possibly higher) categories to loop spaces. In the B-model, the fundamental statement is a sort of Hochschild-Kostant-Rosenberg theorem (see for instance \cite{BZN-connections}*{Corollary 1.21}) identifying Hochschild homology of the category $\QCoh(X)$ of quasicoherent sheaves on $X$ with the space of functions on the (Betti) loop space $\cL_B X$; a categorical analogue can be found in \cite{BZFN}. As we described above, this theorem was subsequently modified in \cite{BZNP} to allow for a nontrivial singular-support Lagrangian.

In the A-model, the fundamental result is Jones's theorem \cite{Jones} relating Hochschild homology of the category $\Loc(M)$ of local systems on $M$ to the cohomology of the free loop space $\fL M,$ which has Fukaya-categorical incarnations as the calculation of symplectic homology of a cotangent bundle \cites{Weber,AS06,Abouzaid-based}. A partially-wrapped (i.e., incorporating a nonzero singular-support Lagrangian) version of this theorem is not yet in the literature, but has a conjectural relation to the {\em microlocal homology} of the loop space (for which see \cite{Schefers}*{\S 1.4}). 
These theorems are what our \Cref{mainconj:perv-traces} aims to categorify, with the 2-category of perverse schobers standing in for a conjectural partially wrapped 2-category of holomorphic Lagrangians in $T^*X.$ 

In this 2-category, we expect that the singular-support Lagrangian $\LL\subset T^*X$ should determine a complex Hamiltonian function $H:T^*X\to \CC$ controlling the wrapping of Lagrangians, and that the contribution of this function to the symplectic action functional on the free loop space $\fL T^*X = T^*\fL X$ can be used to produce a superpotential $W:T^*\fL X\to \CC.$ The right-hand side of the conjectural equivalence \eqref{eq:conjecture} is a sheaf-theoretic model for the Fukaya-Seidel type category associated to the Landau-Ginzburg model $(T^*\fL X, W).$ Examples and more detailed justifications will be postponed to future work which begins to construct this 2-category.

Finally, we point out that the equivalence of categories described in \eqref{eq:mainthm1} is t-exact for the obvious t-structures on both sides --- which may seem odd from the perspective of geometric representation theory, where the canonical t-structure on the left-hand side of \eqref{eq:mainthm1} is the {\em perverse} t-structure. The perverse t-structure is more natural from the perspective of \Cref{mainconj:perv-traces} as well, since perverse schobers by construction are intended to categorify perverse sheaves. One can therefore ask what t-structure on the category $\IndCoh(\cL_B(\CC/\CC^\times))$ corresponds to the perverse t-structure. Conjecturally, the answer is the {\em perverse coherent} t-structure of \cite{AB10} (see also \cite{Cautis-Williams} for further discussion, and \cite{Schnell} for another example of this phenomenon). Our \Cref{mainconj:perv-traces} therefore offers some justification for the perverse coherent t-structure within the framework of 3d mirror symmetry: if one begins with a 3d mirror pair as in \Cref{thm:3dmirror-basic}, the perverse coherent t-structure on the spectral trace category of \Cref{thm:bside-traces} should be equivalent to a perverse t-structure on the automorphic trace category of \Cref{mainconj:perv-traces}.
%

\subsection{Categorical conventions}
Throughout this paper, we work in the setting of stable $\CC$-linear $(\infty,1)$-categories as in \cites{Lurie-HTT,Lurie-HA}, or equivalently pretriangulated $\CC$-linear dg-categories as in \cite{Toen-Morita}. (More general coefficients are possible.) We refer to these simply as ``categories.'' We work model-independently; in fact, the categories we study are all equipped with natural t-structures and the functors we describe are t-exact, so our statements and proofs may be read at the level of abelian (1,1)-categories without losing any essential information.
\subsection{Acknowledgements}
BG is supported by an NSF Postdoctoral Research Fellowship, DMS-2001897.
JH is grateful to Sam Raskin for discussions about the de Rham version of this conjecture, and Tudor Dimofte for conversations about 3d gauge theories.
This research was supported in part by Perimeter Institute for Theoretical Physics. Research at Perimeter Institute is supported by the Government of Canada through the Department of Innovation, Science and Economic Development Canada and by the Province of Ontario through the Ministry of Research, Innovation and Science.

\section{Proof of \Cref{mainthm:betti-tate}}
The remainder of this paper is devoted to the proof of \Cref{mainthm:betti-tate}. We begin in \Cref{subsec:placid} by reviewing the theory of constructible cosheaves on infinite-type Ind-schemes. We then describe the automorphic category $\Sh^!_\cS(\fL\CC)$ as a colimit of categories of constructible (co)sheaves on finite-type stratified spaces. In \Cref{subsec:spectral}, we study the spectral category $\IndCoh(\spectral).$ Finally, in \Cref{subsec:functor}, using the description of $\Sh^!_\cS(\fL\CC),$ we define a functor $\Sh^!_\cS(\fL\CC)$ by defining a compatible system of functors on the constructible-sheaf categories of finite-type approximations to $\fL\CC.$ Once constructed, the verification that this functor is an equivalence of categories is straightforward.

\subsection{Cosheaves on a placid ind-scheme}\label{subsec:placid}
We begin by defining
the left-hand category in the equivalence \eqref{eq:mainthm1}.
The problem of defining automorphic categories associated to infinite-type spaces has long been studied in geometric representation theory, going back at least to \cite{Beilinson-Drinfeld}*{\S 7.11} in the D-module setting and more recently codified in \cite{Raskin-placid}; a parallel theory for perverse sheaves has appeared in \cite{BKV}. To define our constructible (co)-sheaf category, we will need very little of this theory beyond the notion of a placid ind-scheme, as dicussed in \cite{Raskin-placid}*{\S 6}. (Since the completion of this paper, a very brief but possibly helpful survey may has appeared in \cite{BZSV}*{\S B.7}, and a suggestion of the relation to the modern understanding of 6-functor formalisms may be found in \cite{Heyer-Mann-6functor}*{Remark 1.3.10}.)

Let $X=\varinjlim X_\alpha$ be an ind-scheme, presented as a colimit of schemes along closed immersions. Suppose moreover that the schemes $X_\alpha$ have compatible presentations $X_\alpha = \varprojlim X_\alpha^\beta$ as limits of finite-type schemes $X_\alpha^\beta$ along smooth affine fibrations. We say that $X$ is stratified if the $X_\alpha^\beta$ are equipped with compatible stratifications, in the sense that for any of the structure maps $X_\alpha^\beta\to X_{\alpha'}^{\beta'},$ the preimage of any stratum of $X_{\alpha'}^{\beta'}$ is a union of strata of $X_{\alpha}^\beta.$ (All stratifications are assumed conical and satisfying the Whitney conditions.)

The category of interest to us will be built out of categories of constructible sheaves on the varieties $X_\alpha^\beta.$
\begin{definition}
For $M$ a finite-dimensional complex variety equipped with a stratification $\cS,$ we write $\Sh_\cS(M)$ for the category of sheaves on $M$ whose restriction to each stratum of $\cS$ is locally constant. Although we do not impose a finite-rank condition on stalks, we still refer to this category as the category of {\em $\cS$-constructible sheaves} on $M$. This is equivalent to the category
\[
\Fun(\Exit(\cS), \Vect_k)
\]
of functors from the exit-path category \cite{Lurie-HA}*{\S A.6} of the stratified space $(M,\cS).$ 

Similarly, we write $\Sh^!_\cS(M)$ for the category of {\em $\cS$-constructible cosheaves} on  $M$, which is equivalent to the category
\[
\Sh^!_\cS(M)\simeq\Fun(\Enter(\cS),\Vect_k)
\]
of functors out of the {\em entrance-path category} $\Enter(\cS):=\Exit(\cS)^{op}.$
\end{definition}

The category of constructible cosheaves is therefore opposite to the category of constructible sheaves. However, on a finite-dimensional space, there is a Verdier-type duality \cite{Curry-dualities} providing an identification $\Sh_\cS(M)\simeq \Sh^!_\cS(M).$ This is precisely analogous to the situation for D-modules, where there are two D-module theories, referred to in \cite{Raskin-placid} as $D^*$ and $D^!,$ admitting the respective functorialities $f_{*,dR}$ and $f^!$ for a map $f:M\to N.$

We now return to the case where $X=\varinjlim X_\alpha$ is a colimit of schemes along closed immersions, where the schemes $X_\alpha = \varprojlim X_\alpha^\beta$ are limits of finite-type schemes $X_\alpha^\beta$ along smooth affine fibrations, and moreover these $X_\alpha^\beta$ are equipped with a compatible stratification, which we call $\cS.$
\begin{definition}
For $X_\alpha=\varprojlim X_\alpha^\beta$ as above,
we define the category of $\cS$-constructible cosheaves on $X_\alpha$ as
the colimit
\[ \Sh^!_\cS(X_\alpha):=\varinjlim_\beta \Sh_\cS(X_\alpha^\beta)\]
along $!$-pullbacks in the presentation of $X_\alpha.$
For $X=\varinjlim X_\alpha^\beta$ as above, we define the category of 
$\cS$-constructible cosheaves on $X_\alpha^\beta$ as the colimit
\[
     \Sh^!_\cS(X) :=\varinjlim_\alpha \Sh_{\cS}(X_\alpha)
\]
along $*$-pushforwards in the presentation of $X.$
\end{definition}

\subsection{Stratification of the loop space}
We are interested in the case where 
\[X=\fL\CC =\left\{\sum_{m\gg -\infty} a_mt^m\mid a_m\in \CC \right\} \]
is the loop space of $\CC.$ It is a colimit of the infinite-dimensional vector spaces
\[ X_k := \left\{ \sum_{m\geq k} a_m t^m\mid a_m\in \CC \right\} \]
equipped with the inclusions $X_k\hookrightarrow X_{k-1}.$ These spaces are themselves compatibly presented as limits of the finite-dimensional vector spaces 
\begin{equation}\label{eq:xkl}
X_k^\ell:= X_k/X_{k+\ell+1}\simeq \left\{\sum_{m=k}^{k+\ell} a_m t^m\mid a_m\in \CC \right\},
\end{equation}
where the vector space $X_k^\ell$ is presented as the quotient of the vector space $X_k$ by the subspace $X_{k+\ell+1}.$ The spaces $X_k^\ell$ are stratified by the degree of the lowest-order term of the Laurent polynomial (with the 0 polynomial being a point stratum), and it is clear that these strata are compatible under the defining maps, so that they define a stratification $\cS$ of $\fL \CC.$ 
(We will also denote by $\cS$ the restrited stratification on each of the finite-dimensional approximations $X_k^\ell.$)

\subsection{The Betti automorphic category}
Following the definition of \Cref{subsec:placid}, the category of interest to us is defined as the colimit
\[
    \Sh^!_{\cS}(\fL\CC) = \varinjlim\varinjlim\Sh^!_{\cS}(X_k^\ell), 
\]
so our first task is to compute the categories $\Sh^!_{\cS}(X_k^\ell)$ associated to the finite-dimensional pieces $X^\ell_k.$

\begin{lemma}\label{lem:tln}
For $0\leq n\leq \ell,$ let $\cT_\ell(n)\subset \CC^{\ell+1}$ be the submanifold
\begin{equation}\label{eq:tln}
\cT_\ell(n) = \{0\}^n\times \CC^\times \times \CC^{\ell-n}
\end{equation}
Write $\cT$ for the stratification of $\CC^{\ell+1}$ with strata $\cT_\ell(n).$
\begin{enumerate}
    \item As a stratified manifold, $(X^\ell_k,\cS)$ is isomorphic to $(\CC^{\ell+1},\cT).$
    \item Under this identification, the projection $X_k^{\ell}\to X_k^{\ell-1}$ corresponds to the projection $\CC^{\ell+1}=\CC^\ell\times \CC\to \CC^\ell.$ This map sends $\cT_\ell(n)$ onto $\cT_{\ell-1}(n)$ for $n<\ell,$ and $\cT_{\ell}(\ell)$ is mapped to $\cT_{\ell-1}(\ell-1)=0.$
\end{enumerate}
\end{lemma}
\begin{proof}
The identification of $X_k^\ell,$ as defined in \eqref{eq:xkl}, with $\CC^{\ell+1}$ is given by taking $a_k,\ldots,a_{k+\ell}$ as coordinates. Part (b) then follows immediately.
\end{proof}

The following result is the main calculation of this paper.
\begin{proposition}\label{prop:finite-automorphic}
\begin{enumerate}
    \item The category $\Sh^!_{\cS}(X_k^\ell)$ is equivalent to the category $\Mod_{A_k^\ell}$ of $A_k^\ell$-modules, where $A_k^{\ell}$ is is obtained from the path algebra of the quiver
    \begin{equation}\label{eq:quiver}
    \cQ_k^\ell:=\left(\begin{tikzcd}
    \bullet_k\ar{r}{c_{k+1}}\ar[loop above]{}{m_k^\pm}&
    \bullet_{k+1}\ar{r}{c_{k+2}}\ar[loop above]{}{m_{k+1}^\pm}&
    \cdots\ar{r}{c_{k+\ell-1}}&
    \bullet_{k+\ell-1}\ar[loop above]{}{m_{k+\ell-1}^\pm}\ar{r}{c_{k+\ell}}&
    \bullet_{k+\ell}
    \end{tikzcd}\right)
    \end{equation}
    by the ideal of relations 
    \begin{equation}\label{eq:relations-ideal}
    \cI_k^\ell:= \left\langle c_{i+1}(1-m_i), (1-m_j)c_j\right\rangle_{k\leq i <k+\ell, k+1\leq j<\ell},
    \end{equation} and the notation $m^\pm$ means that the map $m$ is required to be invertible.
    \item Let $p_k^{\ell+1}:X_k^{\ell+1}\to X_k^{\ell}$ be the natural quotient map.
    Then in terms of the above quiver presentations, the pullback
    \[
    (p_k^{\ell+1})^!:\Sh^!_\cS(X_k^{\ell})\to \Sh^!_\cS(X_k^{\ell+1})
    \]
    is given by
    \[
    \left(
    \begin{tikzcd}
    V_{k}\ar{r}{c_{k+1}}\ar[loop above]{}{m_k}&
    \cdots\ar{r}{c_{k+\ell-2}}&
    V_{k+\ell-1}\ar[loop above]{}{m_{k+\ell-1}}\ar{r}{c_{k+\ell}}&
    V_{k+\ell}
    \end{tikzcd}
    \right)
    \mapsto
    \left(
        \begin{tikzcd}
    V_{k}\ar{r}{c_{k+1}}\ar[loop above]{}{m_k}&
    \cdots\ar{r}{c_{k+\ell-1}}&
    V_{k+\ell-1}\ar[loop above]{}{m_{k+\ell-1}}\ar{r}{c_{k+\ell}}&
    V_{k+\ell}\ar{r}{\id_{V_{k+\ell}}}\ar[loop above]{}{\id_{V_{k+\ell}}}&
    V_{k+\ell}
    \end{tikzcd}
    \right).
    \]
    \item Let $i_{k}^{\ell}:X_{k}^\ell\to X_{k-1}^{\ell+1}$ be the natural inclusion map. Then in terms of the above quiver presentations, the pushforward
    \[
    (i_{k}^{\ell})_!:\Sh^!_\cS(X_{k}^\ell)\to \Sh^!_\cS(X_{k-1}^{\ell+1})
    \]
    is given by
    \[
    \left(
    \begin{tikzcd}
    V_{k}\ar{r}{c_{k+1}}\ar[loop above]{}{m_k}&
    \cdots\ar{r}{c_{k+\ell-2}}&
    V_{k+\ell-1}\ar[loop above]{}{m_{k+\ell-1}}\ar{r}{c_{k+\ell}}&
    V_{k+\ell}
    \end{tikzcd}
    \right)
    \mapsto
    \left(
    \begin{tikzcd}
    0\ar{r}{0}\ar[loop above]{}{0}&
    V_{k}\ar{r}{c_{k+1}}\ar[loop above]{}{m_k}&
    \cdots\ar{r}{c_{k+\ell-2}}&
    V_{k+\ell-1}\ar[loop above]{}{m_{k+\ell-1}}\ar{r}{c_{k+\ell}}&
    V_{k+\ell}
    \end{tikzcd}
    \right).
    \]
\end{enumerate}
\end{proposition}
\begin{proof}
Because sheaves are more familiar than cosheaves, we will work here with the exit-path rather than entry-path category, and pass to opposites at the end. From the previous lemma, we see that our main task is to understand the geometry of the stratified space $(\CC^{\ell+1},\cT_\ell)$ described in \eqref{eq:tln}. As we shall see, the normal geometry of each stratum (with the slight exception of the smallest stratum) is equivalent, so that the case $\ell=2$ will already contain the entire complexity of the situation.

Before considering $\ell=2$, we begin as a warm-up with the case $\ell=1.$ In this case, the exit-path category $\Exit(\CC,\cT_1)$ contains two isomorphism classes of objects, corresponding to the two strata $\{0\}$ and $\CC^\times$; the endomorphisms of the former are trivial, and the endomorphisms of the latter are generated by a single invertible element $m$. There is an $S^1$'s worth of exit-paths from $\{0\}$ to $\CC^\times,$ or equivalently there is a monodromy-invariant exit path $\{0\}\to\CC^\times.$
In other words, the category $\Sh_{\cT_1}(\CC)$ is equivalent to the data of a pair of vector spaces $V,W$ with a monodromy automorphism $m\in \Aut(W)$ and a monodromy-invariant ``restriction'' map $r:V\to W$, or equivalently a representation of the quiver
\begin{equation}\label{eq:quiver-basic}
    \left(
    \begin{tikzcd}
    \bullet\ar{r}{r}&
    \bullet\ar[loop above]{}{m^\pm}
    \end{tikzcd}
    \right)
\end{equation}
satsifying the monodromy-invariance relation $mr-r=0.$ (In the more familiar perverse description, we keep the nearby-cycles space $W$ but replace $V$ with the vanishing-cycles space $\Cone(r).$)

We now move to the $\ell=2$ case. The first two strata $\{0\}\times \{0\}$ and $\{0\}\times \CC^\times$ are the same as before, with the difference being the addition of the new stratum $\CC^\times \times \CC,$ which, like the second stratum, has a single monodromy automorphism $m_2.$ There is now a new restriction map from the second to the third stratum, which is invariant for the monodromies on both the second and the third strata. As we saw in the case $\ell=1,$ these monodromy-invariances are expressed in the description of $\Sh_{\cT_2}(\CC^2)$ as representations of the quiver
\[
    \left(
    \begin{tikzcd}
    \bullet\ar{r}{r_1}&
    \bullet\ar[loop above]{}{m^\pm_1}\ar{r}{r_2}&
    \bullet\ar[loop above]{}{m^\pm_2}
    \end{tikzcd}
    \right)
\]
by the relations $r_2m_1-r_2=0,$ $m_2r_2-r_2=0$ (together with the previous monodromy-invariance relation $m_1r_1-r_1=0$).

Continuing to the general case, we see that $\Exit(\CC^\ell,\cT_\ell)$ is given by a single contractible stratum and a series of strata homotopic to $S^1,$ with subsequent restriction maps which are invariant for both monodromies. This situation is mostly symmetric: a representation of $\Enter(\CC^\ell,\cT_\ell)$ is given by a sequence of corestriction maps $c_n,$ whose invariance under monodromies $m_n^\pm$ imposes the relations \eqref{eq:relations-ideal}. This proves (1).

Now (2) follows straightforwardly from part (2) of \Cref{lem:tln}: the projection 
collapses the smallest two strata of $\CC^{\ell+1}$, 
and maps the remaining strata isomorphically. 
(Note that the pullback functor is t-exact because the Verdier duality exchanging sheaves and cosheaves with which we began this proof exchanged the $p^!$ pullback with $p^*.$)
Finally, (3) is immediate from the fact that $i_k^\ell$ is an isomorphism onto the complement of the biggest stratum of $\CC^{\ell+1}.$
\end{proof}

\begin{example}
Although the language of exit- and entry-path categories provides the cleanest description of the categories $\Sh^!_{\cS}(X_k^\ell),$ it may obscure some of the actual geometry underlying these categories. In order to clarify the calculations of \Cref{prop:finite-automorphic} for readers more accustomed to concrete sheaf-theoretic calculations, we provide a hands-on calculation of this form in the simplest case, $\ell=1.$

It is a corollary of the standard recollement relations that if $\cS$ is a stratification of $X$ with contractible strata, then the category $\Sh_{\cS}(X)$ is generated by $!$-pushforwards of constant sheaves on strata. (See \cite{Favero-Huang} for applications of this result in mirror symmetry.) Similarly, when the strata are not contractible, the category may be generated by $!$-pushforwards of the ``universal local system'' on each stratum. In general this local system will not be concentrated in degree 0, but in our case, when each noncontractible stratum is homotopic to the $K(\pi,1)$ space $S^1,$ the universal local system $\univ$ has stalk $\kk[m^\pm],$ with monodromy given by $m.$

In the $\ell=1$ case of \cref{prop:finite-automorphic}, the category of interest, namely the category $\Sh_{\CC\cup T^*_0\CC}(\CC)$ of sheaves on the stratified space $\CC=\CC^\times\sqcup\{0\},$ is generated by the skyscraper sheaf $\delta_0=i_!\ul\kk$ and the pushforward $j_!\univ$ of the universal local system, where we write $\{0\}\xhookrightarrow{i} \CC\xhookleftarrow{j} \CC^\times$ for the respective inclusions of strata. The category may therefore be presented as modules over the algebra
\begin{equation}\label{eq:sheaves}
    A:=\End(i_!\ul\kk\oplus j_!\univ) \simeq
    \left(\begin{array}{ll}
        \kk & 0\\
        \Gamma_0(\CC,j_!\univ)& \kk[m^\pm]
    \end{array}\right),
\end{equation}
where we write $\Gamma_0(j_!\univ)$ for sections of the sheaf $i^!j_!\univ,$ which may be computed using the exact sequence 
\[
0\to \Gamma_0(\CC,j_!\univ)\to \Gamma(\CC,j_!\univ)\to \Gamma(\CC^\times,j^*j_!\univ)\to 0
\]
whose middle term is 0 and whose right-hand term is is the cohomology of the universal local system on $\CC^\times.$ We conclude that the bottom-left corner of \eqref{eq:sheaves} is given by the complex $\kk[m^\pm]\xrightarrow{m-1}\kk[m^\pm],$
thus recovering the quiver-with-relations description from \eqref{eq:quiver-basic}.
\end{example}

\begin{remark}
\Cref{prop:finite-automorphic} gives a complete description of the categories $\Sh^!_\cS(X_k)$ and $\Sh^!_\cS(\fL),$ since these are obtained as filtered colimits of the categories $\Mod_{A_k^\ell}$ under the functors described. Informally, we may describe $\Sh^!_\cS(X_k)$ as obtained from $\Mod_{A_k^\ell}$ by ``extending the quiver infinitely to the left'' and requiring that a representation of the quiver be eventually 0 on the left. Afterward, $\Sh^!_\cS(\fL)$ is obtained from the resulting category by ``extending the quiver infinitely to the right'' -- so that in particular there is no longer any terminal node -- and requiring that the arrows $c_n$ are eventually isomorphisms for $n\gg 0.$
\end{remark}

%

\subsection{The Betti spectral category}\label{subsec:spectral}
To complete the proof of \Cref{mainthm:betti-tate}, we need to identify the colimit of the categories described in \Cref{prop:finite-automorphic} as the category of Ind-coherent sheaves on the Betti loop space $\cL_B(\CC/\CC^\times).$
\begin{lemma}\label{lem:quotient-presentation}
The loop space $\cL_B(\CC/\CC^\times)$ admits a global quotient presentation
\begin{equation}\label{eq:globalquotient}
    \cL_B(\CC/\CC^\times) \simeq \left(\Spec\  \CC[x,y^\pm]/(x(y-1)) \right)/\CC^\times,
\end{equation}
where the weights for the $\CC^\times$ action are $|x|=1,|y|=0.$
\end{lemma}
\begin{proof}
The loop space of a global quotient stack $X/G$ may be described as the fiber product
\[
\cL_B(X/G) \simeq (X/G)\times_{(X\times X)/G}(X\times G)/G,
\]
where the left map is the diagonal and the right map is $(a,p),$ where $a$ is the action and $p$ is the projection. In other words, the loop space $\cL_B(X/G)$ is the $G$-quotient of $\{(x,g)\mid gx=x\}.$ In the case where $X=\CC,G=\CC^\times,$ this space is
\[
\{(z,g)\in \CC\times \CC^\times\mid gz = z\} = \CC\times \{1\}\cup_{(0,1)}\{0\}\times \CC^\times.
\]
This is precisely the desired presentation of $\cL_B(\CC/\CC^\times)$
\end{proof}
\begin{remark}
\Cref{lem:quotient-presentation} presents $\cL_B(\CC/\CC^\times)$ as a pushout
\begin{equation}\label{eq:spectral-colimit}
    \cL_B(\CC/\CC^\times) \simeq \varinjlim\left(\CC/\CC^\times \gets B\CC^\times \to \CC^\times \times B\CC^\times\right).
\end{equation}
Since $\IndCoh$ satisfies descent for the colimit \eqref{eq:spectral-colimit}, by the standard trick \cite{GR-vol1}*{Chapter 1, Proposition 2.5.7} of exchanging limits and colimits of presentable categories by passing to left adjoints, the colimit presentation \eqref{eq:spectral-colimit} entails a description of the spectral category $\IndCoh(\cL_B(\CC/\CC^\times))$ as the category of triples $(F^\bullet V, W^\bullet\circlearrowleft\varphi,(W^\bullet)^\varphi\simeq \gr^\bullet(V))$ of a filtered vector space $F^\bullet V,$ a graded vector space $W^\bullet$ equipped with an automorphism $\varphi\in\Aut_{{\sf grVect}}(W^\bullet),$ and an equivalence (of graded vector spaces) between the 1-eigenspace of $\varphi$ and the associated graded of $F^\bullet V.$
\end{remark}

\begin{notation}
At this point it will be useful for us to introduce notation for some objects of the spectral category $\IndCoh(\cL_B(\CC/\CC^\times))$ which we will need to discuss frequently. 
\begin{itemize}
\item $\cO$ denotes the structure sheaf of $\cL_B(\CC/\CC^\times).$
\item $\cO_{\CC/\CC^\times}$ denotes the pushforward of the structure sheaf of $\CC/\CC^\times$ along the inclusion $\CC/\CC^\times\hookrightarrow \cL_B(\CC/\CC^\times)$ of constant loops. 
\item We write $\cO(n)$ for the pullback of the weight-$n$ invertible sheaf $\cO_{B\CC^\times}(n)\in\Coh(B\CC^\times)$ along the map $\cL_B(\CC/\CC^\times)\to B\CC^\times$ coming from the global quotient presentation \eqref{eq:globalquotient}.
\item More generally, any sheaf $\cF\in \IndCoh(\spectral),$ we write $\cF(n)$ for the tensor product $\cF\otimes \cO(n),$ and we refer to this as the {\em twist} of $\cF.$
\end{itemize}
\end{notation}

\begin{proposition}\label{prop:generation}
The category $\IndCoh(\cL_B(\CC/\CC^\times))$ is generated under colimits by the objects $\cO,\cO_{\CC/\CC^\times}$ and their twists.
\end{proposition}
\begin{proof}
Let $A = \CC[x,y^\pm]/(x(y-1))$ and $X=\Spec(A).$ By \Cref{lem:quotient-presentation}, it is sufficient to check that $\IndCoh(X)$ is generated under colimits by the $A$-modules $A$ and $A/(y-1)$ or in other words that these two objects generate the category $\Coh(X)$ of coherent sheaves on $X$. The category $\Perf(X)$ is generated by the object $A$, and it is known (see for instance \cite{Dyckerhoff-gen}) that the singularity category $\Coh(X)/\Perf(X)$ of the nodal curve $X$ is generated by the skyscraper sheaf $A/(x,y-1)$ at the node. Since the skyscraper sheaf $\delta$ may be generated from $A$ and $A/(y-1),$ we conclude that the module $A/(y-1)$ also generates the singularity category, and therefore that $A$ and $A/(y-1)$ jointly generate $\Coh(X).$
\end{proof}

\subsection{Constructing a functor}\label{subsec:functor}
We are now ready to write down a functor which will implement the equivalence \eqref{eq:mainthm1}. As the automorphic category $\Sh^!(\fL\CC)$ is presented as a colimit
\[
\Sh^!(\fL\CC) = \varinjlim_{k,\ell}\Sh^!(X_k^\ell),
\]
we will construct this functor by writing down a compatible family of functors
\[
f_{k,\ell}:\Sh^!(X_k^\ell)\to\IndCoh(\spectral).
\]

Recall that the category $\Sh^!_\cS(X_k^\ell)$ is presented by the algebra $A_k^\ell$ obtained from the quiver $\cQ_k^\ell$ defined in \eqref{eq:quiver} by imposing the relations $\cI_k^\ell$ defined in \eqref{eq:relations-ideal}.
%
%
To define a functor with domain $\Sh^!(X_k^\ell),$ it is therefore sufficient to give a $\cQ_k^\ell$-diagram of objects and maps in $\Ind\Coh(\spectral)$ satisfying the relations $\cI_k^\ell$: the image of the vertex $v_n$ in the quiver $\cQ_k^\ell$ determines the image of the projective object $P_n$ corresponding to $v_n.$
\begin{definition}
We write $f_{k,\ell}:\Sh^!_\cS(X_k^\ell)\to\IndCoh(\spectral)$ for the functor defined by the $\cQ_k^\ell$-diagram
\begin{equation}\label{eq:fkl}
    \left(\begin{tikzcd}
    \cO(k)\ar[loop above]{}{y}\ar{r}{x}&
    \cO(k+1)\ar[loop above]{}{y}\ar{r}{x}&
    \cdots\ar{r}{x}&
    \cO(k+\ell-1)\ar[loop above]{}{y}\ar{r}{x}&
    \cO_{\CC/\CC^\times}(k+\ell)
    \end{tikzcd}\right)
\end{equation}
in $\IndCoh(\spectral),$ where the satisfaction of the relations $\cI_k^\ell$ comes from the relation $x(y-1)$ in \eqref{eq:globalquotient}.
\end{definition}

\begin{lemma}\label{lem:fullfaith}
The functors $f_{k,\ell}$ are fully faithful.
\end{lemma}
\begin{proof}
The category $\Sh^!(X_k^\ell)$ is the derived category of an abelian category, namely the abelian category of representations of the quiver with relations described in \Cref{prop:finite-automorphic}, generated by the projective objects associated to vertices of the quiver $\cQ_k^\ell,$ so
it is sufficient to check full faithfulness on these projective objects. The Homs among the projective objects 
are freely generated by the arrows $c_n$ and $m_n$ (the latter invertible) subject to the relations $m(c-1) = (c-1)m = 0$ (with the appropriate indexing). The same is true of the Hom spaces among the $\cO(n),\cO_{\CC/\CC^\times}(n).$
\end{proof}

\begin{lemma}\label{lem:compatibility}
The functors $f_{k,\ell}$ commute with the functors 
$(p_k^\ell)^!$ and $(i_k^\ell)_!$ described in \Cref{prop:finite-automorphic}.
\end{lemma}
\begin{proof}
Index the vertices in the quiver $\cQ_k^\ell$ by $k,\ldots,k+\ell,$ as in \eqref{eq:quiver}. We will write $P_n$ for the projective object of $\Mod_{A_k^\ell}$ at the $n$th vertex. Similarly, we denote the diagram \eqref{eq:fkl} by $\cD_k^\ell,$ and we write $\cD_k^\ell(n)$ for the object placed at the $n$th vertex.

Observe from \Cref{prop:finite-automorphic} that the pullback $(p_k^\ell)^!$ sends $P_n\mapsto P_n$ for $k\leq n<k+\ell,$ and $P_{k+\ell}\mapsto P_{k+\ell}/(m_{k+\ell}-1).$ This is precisely the intertwining relation between the maps $f_{k,\ell}$ and $f_{k,\ell+1}.$
Similarly, the pushforward $(i_k^\ell)_!$ takes the projective $P_n$ to $P_{n+1},$ for all $n\in\{k,\ldots,k+\ell\}.$ Compatibility with the $f_{k,\ell}$ is therefore the requirement that $\cD_k^\ell$ embed in $\cD_{k-1}^{\ell+1}$ as the subquiver spanned by the last $\ell+1$ vertices, which is the case.
\end{proof}
\begin{definition}\label{def:the-functor}
As a result of \Cref{lem:compatibility}, the functors $f_{k,\ell}$ assemble into a functor whose domain is the colimit $\varinjlim\Sh^!_\cS(X_k^\ell) = \Sh^!_\cS(\fL \CC).$ We denote this functor by 
\[
F:\Sh^!_\cS(\fL \CC) \to \IndCoh(\spectral).
\]
\end{definition}

\begin{theorem}
The functor $F$ of \Cref{def:the-functor} is an equivalence of categories.
\end{theorem}
\begin{proof}
For full faithfulness, we apply the following fact, which can be found as \cite{HR21}*{Lemma 8.4.0.2}: suppose $\cC = \varinjlim \cC_i$ is a filtered colimit of strongly continuous fully faithful functors of compactly generated categories, and $\cC\to \cD$ is a functor preserving compact objects. If each $\cC_i\to \cD$ is fully faithful, then $\cC\to \cD$ is fully faithful. In our case, the functors $\cC_i\to \cD$ are the functors $f_{k,\ell},$ whose full faithfulness was checked as \Cref{lem:fullfaith}.

Essential surjectivity follows from the fact that the projectives in the quiver $\cQ_k^\ell$ map to twists of $\cO$ and $\cO_{\CC/\CC^\times},$ and twists of those objects jointly generate $\IndCoh(\spectral)$ by \Cref{prop:generation}.
\end{proof}

To complete the proof of \Cref{mainthm:betti-tate}, it remains to establish compatibility with Betti geometric class field theory.

\begin{proposition}
The equivalence of categories $F$ constructed above intertwines convolution by local systems on $\Loc(\fL \CC^\times)$ with tensor product by coherent sheaves pulled back from $\IndCoh(\cL_B(B\CC^\times)).$
\end{proposition}
\begin{proof}
First, observe that $F$ is a t-exact functor (for the natural t-structures defined on both sides of the equivalence), and by passing to hearts, it is sufficient to check the statement of the proposition at the level of abelian categories: this follows from \cite{Gannon-nondegen}*{Theorem A.4}, which shows that passing to hearts defines a symmetric monoidal equivalence between a category of Grothendieck abelian categories with enough projective objects, and a category of derived categories of Grothendieck abelian categories with enough projectives.

Now note that an action of the abelian heart of the equivalent symmetric monoidal categories of \Cref{eq:betti-gcft} on an abelian category $\cC$ is determined by the data of a $\CC[a^\pm]$-linear structure on $\cC$ together with a $\CC[a^\pm]$-linear autoequivalence $\cC\xrightarrow{\sim}\cC,$ and it is straightforward to check that these agree.

%
Namely, on the left-hand side of \eqref{eq:betti-gcft}, convolution by the local system with monodromy $a$ on the degree-zero component $\CC^\times\times \{0\}\subset \CC^\times\times \ZZ\simeq \fL\CC^\times$
picks out the monodromy-$a$ part of local systems on the strata of $\fL\CC$; under the equivalence $F,$ this is tensoring with the (placed-in-weight-0) skyscraper sheaf $\delta_a\in\Coh(\CC^\times\times B\CC^\times).$ On the other hand, convolution by the universal local system $\univ[n]$ on the $n$th component of $\fL\CC=\varinjlim X_k$ by $n$ iterations of the pushforward $X_k\hookrightarrow X_{k-1}$ (which makes sense for $n$ negative because of the colimit); as desired, this corresponds on the spectral side to the tensor product with $\cO_{\CC^\times}(n),$ the structure sheaf of $\CC^\times$ twisted by weight $n$.
\end{proof}

\bibliographystyle{plain}
\bibliography{refs}
\end{document}